\documentclass[12pt,a4paper]{amsart}
\usepackage{latexsym}
\usepackage{amsmath,amssymb,exercise}
\usepackage{dsfont}

\usepackage[bookmarks=true,hyperindex,pdftex,colorlinks,citecolor=blue]{hyperref}

 \newtheorem{theo}{Theorem}[section]

 \newtheorem{example}[theo]{Example}
  \newtheorem{lem}[theo]{Lemma}

\newtheorem{propos}[theo]{Proposition}

\newtheorem{cor}[theo]{Corollary}

\theoremstyle{definition}
 \newtheorem{definition}[theo]{Definition}
 \newtheorem{prob}[theo]{Problem}

\newcommand{\R}{{\mathbb{R}}}

\newcommand{\N}{{\mathbb{N}}}

\newcommand{\F}{\mathfrak{F}}

\newcommand{\eps}{\varepsilon}

\newcommand{\del}{\delta}

\newcommand{\diam}{\mathrm{diam}}

\newcommand{\bea}{\begin{eqnarray*}}
\newcommand{\eea}{\end{eqnarray*}}
\newcommand{\beq}{\begin{eqnarray}}
\newcommand{\eeq}{\end{eqnarray}}

\renewcommand{\leq}{\leqslant}
\renewcommand{\geq}{\geqslant}
\renewcommand{\le}{\leqslant}
\renewcommand{\ge}{\geqslant}

\numberwithin{equation}{section}
\begin{document}
\title[Plastic pairs]{Plastic pairs of metric spaces}

\author[Kadets]{V. Kadets}

\address[Kadets]{ \href{http://orcid.org/0000-0002-5606-2679}{ORCID: \texttt{0000-0002-5606-2679}} {Department of Mathematics Holon Institute of Technology (Israel) and School of Mathematics and Informatics, V.N.Karazin Kharkiv National University (Ukraine) }}
\email{vova1kadets@yahoo.com}

\author[Zavarzina]{O. Zavarzina}
\address[Zavarzina]{\href{http://orcid.org/0000-0002-5731-6343}{ORCID: \texttt{0000-0002-5731-6343}} School of Mathematics and Informatics, V.N. Karazin Kharkiv National University (Ukraine)}
\email{olesia.zavarzina@yahoo.com}

\thanks{ On the early stage of the project, the authors were supported  by the  National Research Foundation of Ukraine funded by Ukrainian State budget in frames of project 2020.02/0096 ``Operators in infinite-dimensional spaces:  the interplay between geometry, algebra and topology''. The first author was supported by Weizmann Institute of Science (Israel) Emergency program for scientists affected by the war in Ukraine}

\begin{abstract}
We address pairs $(X, Y)$ of metric spaces with the following property: for every mapping $f: X \to Y$ the existence of points $x, y \in X$ with  $d(f(x),f(y)) > d(x,y)$ implies the existence of $\widetilde{x}, \widetilde{y}\in X$ for which $d(f(\widetilde{x}),f(\widetilde{y})) < d(\widetilde{x},\widetilde{y})$. We give sufficient conditions for this property and for its uniform version in terms of finite $\eps$-nets and finite $\eps$-separated subsets.
\end{abstract}

\maketitle
\section{Introduction}

\begin{definition} \label{def-plast}  A metric space $X$ is said to be \emph{Expand-Contract plastic} (or simply, an EC-plastic space) if every noncontractive bijection from $X$ onto itself is an isometry.
\end{definition}

This property was introduced and studied in depth in \cite{NaiPioWing}. The initial point of that study was the fact, that every totally bounded metric space is EC-plastic \cite{FHuDeh}. In reality, totally bounded spaces possess a stronger property  \cite[Theorem 1.1]{NaiPioWing} or \cite[Satz IV]{FHuDeh}, which for our convenience we formalize in the following definition.

\begin{definition} \label{def-str-plast}  A metric space $X$ is said to be \emph{strongly plastic} if for every mapping $f: X \to X$ the existence of points $x, y \in X$ with  $d(f(x),f(y)) > d(x,y)$ implies the existence of $\widetilde{x}, \widetilde{y}\in X$ for which $d(f(\widetilde{x}),f(\widetilde{y})) < d(\widetilde{x},\widetilde{y})$.
\end{definition}

In other words, $X$ is strongly plastic if every noncontractive mapping $f: X \to X$  is an isometric embedding (not necessarily bijective).

There are other examples of EC-plastic spaces, in particular the closed unit ball $B_X$ of every strictly convex Banach space $X$ is  EC-plastic \cite[Theorem 2.6]{CKOW2016}. It is an open problem whether  EC-plasticity is possessed by unit balls of all Banach spaces.

\begin{example} \label{example1-not-str-plast} The unit ball $B_{\ell_2}$ of the Hilbert space $\ell_2$ is not strongly plastic. So, it is a natural example of EC-plastic space that is not strongly plastic.
\end{example}
\begin{proof}
Indeed, denote, for every $x=(x_1, x_2, \ldots) \in B_{\ell_2}$,
$$
f(x) = \left(\sqrt{1 - \|x\|^2}, x_1, x_2, \ldots \right).
$$
This defines non-contractive mapping  $f: B_{\ell_2} \to B_{\ell_2}$ which is not an isometry, because for $x= (0, 0, \ldots)$ and $y = (1,0, 0, \ldots)$ we have
$$
\left\| f(x) - f(y) \right\| = \sqrt{2} > 1 = \left\| x - y \right\|. \qedhere
$$
\end{proof}

The next definition is motivated  by  \cite{Nitka57, Nitka98}.

\begin{definition} \label{def-str-un-plast}  A metric space $X$ is said to be \emph{uniformly strongly plastic} (\emph{uniformly EC-plastic}) if for every $\eps > 0$ there is $\delta = \delta(\eps) > 0$ such that  for every $f\colon X\to X$ (respectively, for every bijection $f\colon X\to X$) if there are $x,y\in X$ with $d(f(x),f(y)) > d(x,y)+\eps$, then there are $\widetilde{x}, \widetilde{y}\in X$ with $d(f(\widetilde{x}),f(\widetilde{y})) < d(\widetilde{x},\widetilde{y})-\delta$.
\end{definition}

 The main result from \cite{Nitka98}, previously demonstrated by the same author in \cite[Satz 1]{Nitka57} under the additional assumption of the denseness of $f(X)$, says that every  totally bounded metric space $X$ is uniformly strongly plastic. Moreover, the estimation of $\del$ relies only on the smallest number $n(\eps)$ of the elements of $\eps$-net in $X$. Namely, Nitka demonstrated that

\begin{equation} \label{eq-nitka1}
\delta = \frac{2 \eps}{11\left(n({\frac{\eps}{11}})\left(n({\frac{\eps}{11}})-1\right) + 2\right)}
\end{equation}
can be used.

Our paper is mainly devoted to analogous questions for mappings between two different metric spaces.

\begin{definition} \label{def-plast-pair}  A pair $(X, Y)$ of metric spaces is said to be \emph{EC-plastic} (\emph{strongly plastic}) if every noncontractive bijection (mapping) from $X$ to $Y$ is an isometry (an isometric embedding, respectively).
\end{definition}

EC-plasticity for pairs of unit balls of different Banach spaces has being studied in \cite{Zav} and developed further  in \cite{AnKaZa}. In particular, it is known that for EC-plasticity of a pair of unit balls it is sufficient to assume strict convexity of one of the corresponding spaces.

The uniform versions of the above properties can be formulated in the following way.

\begin{definition} \label{def-str-un-plast-pairs} Let  $(X, Y)$ be a pair of metric spaces. For every $\eps > 0$ denote $\delta_{X, Y}(\eps)$ the suprema of those $\delta \ge 0$ such that for every bijection $f\colon X\to Y$   if there are $x,y\in X$ with $d(f(x),f(y)) > d(x,y)+\eps$, then there are $\widetilde{x}, \widetilde{y}\in X$ with $d(f(\widetilde{x}),f(\widetilde{y})) < d(\widetilde{x},\widetilde{y})-\delta$. The function $\eps \mapsto \delta_{X, Y}(\eps)$ is said to be the \emph{modulus of plasticity} of the pair  $(X, Y)$. The pair  $(X, Y)$ is said to be \emph{uniformly EC-plastic} if $\delta_{X, Y}(\eps) > 0$ for every  $\eps > 0$.

Analogously, just substituting the word ``bijection'' by the word ``mapping'' we define the \emph{modulus of strong plasticity} $\delta_{X, Y}^s(\eps)$, and we call  $(X, Y)$ \emph{uniformly strongly plastic} if $\delta_{X, Y}^s(\eps) > 0$ for every  $\eps > 0$.
\end{definition}

Our goal is to demonstrate that if the ``size'' of $Y$ is not bigger than the ``size'' of $X$ and the spaces are totally bounded then the pair  $(X, Y)$ is plastic or strongly plastic, depending on the meaning of the word ``size'' above. We are interested also in the uniform versions of such results and the corresponding quantitative estimates.

Remark, that previously known for the case of $X = Y$ approach from \cite{FHuDeh, Nitka57, Nitka98} and \cite{NaiPioWing}  used the dynamical system generated by $f: X \to X$ (the sequence of iterations $f, f \circ f, f \circ f \circ f, \ldots$), which makes no sense for $f: X \to Y$.  So, for our goals we develop another argument.

The structure of the paper is as follows.  At first, in Section \ref{sec-finite} we consider metric spaces that consist of finitely many points. On this example we explain a very simple idea whose variations work later (with some modifications) in more involved cases. We give optimal quantitative estimates for the case when $X=Y$ is finite. Also, we present an example of metric space which is plastic but not uniformly plastic. After that, in Section \ref{sec-precompact} we consider the general case of pairs of infinite totally bounded spaces.

Let $X$ be a metric space, $A\subset X$ be a non-empty subset. Below we use the notation $|A|$ for the number of elements of $A$ (which may be finite or equal to $+\infty$), $\diam (A)$ stands for the diameter of $A$
$$
\diam (A) = \sup\{d(x, y): x,y \in X\}.
$$

Recall that a subset $A$ of a metric space $X$ is said to be an $\eps$-net if for every $x \in X$ there is $a \in A$ with $d(x,a) < \eps$. A metric space  is called \emph{totally bounded} (or precompact) if for every $\eps > 0$ the space possesses a finite $\eps$-net.

\section{Plasticity in finite metric spaces} \label{sec-finite}

In this section we deal with finite metric spaces. If $|X| > |Y|$  then every  $f\colon X\to Y$ is not injective, which implies strong plasticity of $(X, Y)$ (and EC-plasticity as well).

 In the case of $|X| = |Y| < \infty$  EC-plasticity is the same as strong plasticity. Indeed, let $(X, Y)$ be EC-plastic and $f\colon X\to Y$ be a mapping that increases distance between some two points. In the case of $f$ being bijective, EC-plasticity implies that $f$  decreases distance between some other two points (because $F$ is not an isometry and thus, by EC-plasticity, cannot be noncontractive). In the case of $f$ being not bijective, the condition  $|X| = |Y| < \infty$ implies that $f$ is not injective, so again $f$ decreases the distance between some other two points.

Finally, if $N =|X| < |Y| < \infty$, the strong plasticity of $(X, Y)$ is equivalent to the following property: for every metric subspace $E \subset Y$ with $|Y| = N$ the pair  $(X, E)$ is EC-plastic.

All this explains why in this section we temporarily forget about strong plasticity and concentrate on EC-plasticity. Moreover, in the case of spaces $X, Y$ of different finite cardinality $|X| \neq |Y|$ the absence of bijections between $X$ and $Y$ trivially guaranties EC-plasticity of the pair $(X, Y)$. So, the interesting for us case that we consider below in this section is $|X| = |Y| = N$, where $N$ is a natural number.

\subsection{EC-plasticity in the finite case} \label{subsec-finite1} $ $

\medskip

We start with an elementary but enlightening remark. For a finite metric space $A$, consider the following quantity

\beq  \label{eq-sigma-1}
\sigma(A)=\sum_{a,b\in A}d(a,b).
\eeq

\begin{theo} \label{theo-fin-1}
Let $X$, $Y$ be finite metric spaces with $|X| = |Y| = N \in \N$ and $\sigma(Y) \le \sigma(X)$, then the pair  $(X, Y)$ is  EC-plastic. Moreover, $(X, Y)$ is uniformly plastic with  $\delta_{X, Y}(\eps) \ge \eps \left(\frac{N(N-1)}{2} - 1\right)^{-1}$ for all $\eps > 0$.
\end{theo}
\begin{proof}

 Let $f\colon X\to Y$ be a bijection. Fix  $\eps>0$ and assume  the existence of  $x,y\in X$ such that
\begin{equation}\label{theo-fin-1ineq1}
d(f(x),f(y)) > d(x,y)+\eps.
\end{equation}
Denote $\eps \left(\frac{N(N-1)}{2} - 1\right)^{-1}$ by $\delta$. Our goal is to demonstrate the existence of
 $\widetilde{x}, \widetilde{y}\in Y$ with $d(f(\widetilde{x}),f(\widetilde{y})) < d(\widetilde{x},\widetilde{y})-\delta$. Assume to the contrary that
\begin{equation}\label{theo-fin-1ineq2}
d(f(a),f(b)) \ge d(a,b)-\delta
\end{equation}
for all $a, b\in X$. Then, combining bijectivity of $f$ with conditions \eqref{theo-fin-1ineq1} and \eqref{theo-fin-1ineq2}, we deduce that
\begin{align*}
\sigma(Y) &=\sum_{a,b\in X}d(f(a),f(b)) = d(f(x),f(y)) + \sum_{(a,b)\in X \times X \setminus \{(x,y)\}}d(f(a),f(b)) \\
&> d(x, y) +\eps + \sum_{(a,b)\in X \times X \setminus \{(x,y)\}}(d(a,b)-\delta) \\
&= \sum_{a,b\in X}d(a,b) +\eps - \left(\frac{N(N-1)}{2} - 1\right) \delta = \sigma(X),
\end{align*}
which gives the desired contradiction.
\end{proof}

In order to generalize the above argument, let us introduce one definition more.

\begin{definition} \label{def-fin-2} Denote $\mathcal M_N$ the collection of all $N$-point metric spaces. For $X, Y \in \mathcal M_N$, a mapping $f : X \to Y$ is called \emph{expansion} if it is noncontractive and for some pair $a, b\in X$ it increases strictly the distance: $d(f(a),f(b)) > d(a,b)$. A function $\psi : \mathcal M_N \to [0, +\infty)$ is said to be a \emph{proper measurement} if $\psi(Y) > \psi(X)$ for every $X, Y \in \mathcal M_N$ for which exists an expansion $f : X \to Y$.
\end{definition}

With this definition, Theorem \ref{theo-fin-1} can be generalized in the following way:

\begin{propos} \label{prop-fin-3}
If $\psi : \mathcal M_N \to [0, +\infty)$ is a proper measurement then every pair $X, Y \in \mathcal M_N$ with  $\psi(Y) \le \psi(X)$ is EC-plastic.
\end{propos}

Following the proof of Theorem \ref{theo-fin-1}, one can easily see that the quantity $\sigma$ defined by \eqref{eq-sigma-1} is a proper measurement. There is a number of other proper measurements very similar to $\sigma$. Namely, fixing a strictly increasing function $g: [0, +\infty) \to [0, +\infty)$ with $g(0) = 0$ and denoting for every $A \in \mathcal M_N$
\beq  \label{eq-sigma-gg}
\sigma_g(A)=\sum_{a,b\in A} g(d(a,b)),
\eeq
one defines a proper measurement.

It would be interesting to give an explicit description of all proper measurements. Also, the following question looks interesting.

\begin{prob} \label{prob-fin-4}
Is it true that for every EC-plastic pair $X, Y \in \mathcal M_N$ there is a proper measurement $\psi : \mathcal M_N \to [0, +\infty)$ such that $\psi(Y) \le \psi(X)$?
\end{prob}

{\bf Remark.} If for an EC-plastic pair $X, Y \in \mathcal M_N$ there exists a non-contractive bijection $f\colon X \to Y$, then one can choose $\psi(X):=\sigma(X)$. Indeed,
$$\psi(Y)=\sum_{a,b\in Y}d(a,b)=\sum_{x,y\in X}d(f(x),f(y))=\sum_{x,y\in X}d(x,y)=\psi(X),$$
since $f$ is an isometry. So, in fact Problem \ref{prob-fin-4} is open only for trivially EC-plastic pairs  $X, Y \in \mathcal M_N$ for which there are no non-contractive bijections $f\colon X \to Y$.

\subsection{Uniform plasticity in the finite case for $X = Y$} \label{subsec-finite-applic} $ $

\medskip
The aim of this subsection is to give the promised in the Introduction optimal estimate of the modulus of plasticity of a finite metric space through the number of elements of the space. The case of $X = Y$ that we consider now permits to use iterations of a map $f: X \to Y$ and corresponding orbits of elements, which enables us to give better estimates than in Theorem \ref{theo-fin-1}.

For each $N \in \N$ denote $M(N)$ the following quantity:
$$
M(N)=\begin{cases}
    k(k+1), & \mbox{if } N=2k+1, \, k \in \{2, 3, 4, \ldots \} \\
    (2k-1)(2k+1), & \mbox{if } N=4k, \, k \in \{2, 3, 4, \ldots \} \\
    (2k-1)(2k+3), & \mbox{if } N=4k+2, \, k \in \{2, 3, 4, \ldots \} \\
    N, & \mbox{otherwise}.
  \end{cases}
$$
\begin{lem} \label{lem-lcm}
$M(N) = \max\{\max_{l+m\leq N}LCM(l,m),N\}$,  where LCM is the abbreviation for the least common multiple of two naturals.
\end{lem}
\begin{proof}
Let us first compare $N$ and $\max_{l+m\leq N}LCM(l,m)$ for different $N$. It is well known that the function $g(n)=n(N-n)$ reaches its maximum when $n=N/2$. Since we are interested only in natural $n$ and $LCM(n,N-n)=\frac{n(N-n)}{GCD(n, N-n)}$, where GCD is the abbreviation for the greatest common divisor, we have
$$
LCM(n,N-n)=\begin{cases}
                     k(k+1), & \mbox{if } N=2k+1 \\
                     (2k-1)(2k+1), & \mbox{if } N=4k \\
                     (2k-1)(2k+3), & \mbox{if } N=4k+2.
                   \end{cases}
$$
Using this representation, one can also observe $LCM(n,N-n)\geq LCM(n,N-1-n)$ unless $N=6$. Then
$$
\max_{l+m\leq N}LCM(l,m)=\begin{cases}
                     k(k+1), & \mbox{if } N=2k+1 \\
                     (2k-1)(2k+1), & \mbox{if } N=4k \\
                     (2k-1)(2k+3), & \mbox{if } N=4k+2 \\
                     6, & \mbox{if } N=6.
                   \end{cases}
$$
It is easy to see, that $\max_{l+m\leq N}LCM(l,m)>N$ unless $N=3$, $N=4$ or $N=6$. Thus, in fact,  $\max\{\max_{l+m\leq N}LCM(l,m),N\}=M(N)$.
\end{proof}

\begin{theo} \label{theor2}
Let $X$ be a metric space with $|X| \le N \in \N$.  Then  $\delta_{X, X}(\eps) \ge \frac{\eps}{M(N)-1}$ for all $\eps > 0$.
\end{theo}
\begin{proof}
Fix $\eps > 0$ and a bijection $f\colon X \to X$. Assume the existence of $x,y\in X$ such that $d(f(x),f(y)) > d(x,y)+\eps$. Our goal is to demonstrate the existence of $\widetilde{x}, \widetilde{y}\in X$ such that $d(f(\widetilde{x}),f(\widetilde{y})) < d(\widetilde{x},\widetilde{y}) - \frac{\eps}{M(N)-1}$.

Consider the orbits $O(x)$ and $O(y)$ of the elements $x$ and $y$ under the action of $f$ and denote by $D$  the least $k$, such that $f^k(x)=x$ and $f^k(y)=y$. If $O(x) \neq O(y)$ one may conclude $D=LCM(i,j)\leq\max_{l+m\leq N}LCM(l,m)$, where $i$ and $j$ denote the number of elements in $O(x)$ and $O(y)$ respectively. If $O(x) = O(y)$, i.e. $y \in O(x)$ and vice versa, then $D\leq N$.
Now we may conclude $D \leq M(N)$ and argue by contradiction. Suppose for every $a, b\in X$
\begin{equation*}
  d(f(a),f(b)) \ge d(a,b)-\frac{\eps}{M(N)-1}.
\end{equation*}
 This inequality, in particular, implies that for every $k \in \N$
 \begin{align} \label{theor2eq1}
   d(f^{k+1}(x),f^{k+1}(y))& \ge d(f^k(x),f^k(y))-\frac{\eps}{M(N)-1} \nonumber \\
   &\geq d(f^k(x),f^k(y))-\frac{\eps}{D-1}.
 \end{align}
Iterating \eqref{theor2eq1}, we get
$$
d(f^{D}(x),f^{D}(y)) \geq d(f(x),f(y)) - (D-1)\frac{\eps}{D-1} = d(f(x),f(y)) - \eps,
$$
which together with condition $d(f(x),f(y)) > d(x,y)+\eps$  implies  that
$$ d(x,y)= d(f^{D}(x),f^{D}(y)) > d(x,y).$$
So, we obtained the desired contradiction.
\end{proof}
\begin{theo} \label{theo-fin-sharp}
The estimate from the previous theorem cannot be improved.
\end{theo}
\begin{proof}
Since the estimate depends on $N$ and $\eps$, the examples which prove our statement will depend on $N$  and $\eps$ as well.   \\
{\bf CASE 1.} $N\neq 3,4,6$.  Let us introduce
$$
i=\begin{cases}
    k, & \mbox{if } N=2k+1 \\
    2k-1, & \mbox{if } N=4k \\
    2k-1, & \mbox{if } N=4k+2.
  \end{cases}
$$
and
$$
j=\begin{cases}
    k+1, & \mbox{if } N=2k+1 \\
    2k+1, & \mbox{if } N=4k \\
    2k+3, & \mbox{if } N=4k+2.
  \end{cases}
$$
Take a set $X=\{x_1,...x_i, y_1,...,y_j\}$ and introduce a function $f$ on it as follows
$$
f(x)=\begin{cases}
       x_{k+1}, & \mbox{if } x=x_k, 1\leq k < i \\
       x_1, & \mbox{if } x=x_i \\
       y_{k+1}, & \mbox{if } x=y_k, 1\leq k < j \\
       y_1, & \mbox{if } x=y_j,
     \end{cases}
$$
that is, $\{x_1,...x_i\}$ is the orbit of $x_1$,  $\{y_1,...,y_j\}$  is the orbit of $y_1$ under the action of $f$, and $M(N) = ij$.
Now, let us introduce the distance on $X$ by the following rule:
$$
d(x,y)=\begin{cases}
         a, & \mbox{if } x=x_1, y=y_1 \\
         a+\eps, & \mbox{if } x=f(x_1), y=f(y_1) \\
         d(f^{k-1}(x_1),f^{k-1}(y_1))-\frac{\eps}{M(N)-1}, & \mbox{if } x=f^k(x_1)\mbox{ and }\\
         & y=f^k(y_1), 2\leq k < ij \\
         a, & \mbox{otherwise},
       \end{cases}
$$
where $a \in \R$,  $a \geq \eps$ is a fixed number. With this definition, all the pairwise distances belong to the interval $[a, a + \eps] \subset [a, 2a]$, which ensures the validity of the triangle inequality.

For our metric space $(X, d)$ with $|X| \le N$ and the function $f$ the estimate is accurate:  $f$ increases the distance between $x_1$ and $y_1$ on $\eps$ but does not decrease any distance stronger than on $\frac{\eps}{M(N)-1}$.

In order to get an example with $|X| = N$ one may add artificially to the space that is constructed above $N - ij$ fixpoints $\{v_1, \ldots v_{N - ij}\}$ of $f$ and define the missing distances $d(x, y)$ to be $2a$ when at least one of the points $x, y$  belongs to $\{v_1, \ldots v_{N - ij}\}$.

For $N=3,4,6$ take a set $Z=\{z_1,...,z_N\}$ and introduce on it the following cyclic function $g$:
$$
g(x)=\begin{cases}
       z_{k+1}, & \mbox{if } x=z_k, 1\leq k < N \\
       z_1, & \mbox{if } x=z_N.
     \end{cases}
$$
The corresponding metrics in this case is
$$
p(x,y)=\begin{cases}
         a, & \mbox{if } x=z_1, y=g(z_1) \\
         a+\eps, & \mbox{if } x=g(z_1), y=g^2(z_1) \\
         d(g^{k-1}(x_1),g^{k-1}(y_1))-\frac{\eps}{M(N)-1}, & \mbox{if } x=g^k(z_1) \mbox{ and }\\
         & y=g^{k+1}(z_1), 2\leq k < N \\
         a, & \mbox{otherwise},
       \end{cases}
$$
and as before $a \geq \eps$ in order to ensure the triangle inequality.
For the metric space $(X, p)$ and the mapping $g$ the estimate is accurate.
\end{proof}

\subsection{An application: EC-plasticity does not imply uniform plasticity} \label{subsec-finite-applic} $ $

\medskip

The goal of this subsection is to construct an example of strongly plastic (and hence EC-plastic) infinite metric space $X$ which is not uniformly plastic. The idea of construction consists in putting together finite pieces $X_n$, each of which is uniformly plastic, but whose moduli of uniform plasticity at some fixed $\eps > 0$ tend to zero as $n$ tends to infinity.

So, fix  $\eps_0 = 1$ and a vanishing sequence of $\delta_n > 0$. Consider $a_n = 1 + \frac{1}{10 n}$, $n = 1, 2, \ldots$. Then $a_n \geq \eps_0$, so, according to the construction from Theorem \ref{theo-fin-sharp}, for each $n \in \N$ we may (and do) find a finite metric space $X_n$ and a mapping $f_n : X_n \to X_n$ with the following properties.
\begin{enumerate}
\item $d(x,y) \in [a_n, a_n + 1]$ for all $x, y \in X_n$.
\item $d(f_n(x), f_n(y)) \ge d(x,y) - \delta_n$ for all $x, y \in X_n$.
\item There is a pair $x_n, y_n \in X_n$ with $d(f_n(x_n), f_n(y_n)) = d(x,y) + 1$.
\end{enumerate}
  Without loss of generality, we may assume additionally the existence of $e_n \in X_n$ such that $d(x, e_n) = a_n + 1$ for all $x \in X_n \setminus \{e_n\}$. This can be achieved by adding to $X_n$ a new point $e_n$, defining the distances to the old points as $d(x, e_n) = a_n + 1$, and extending $f_n$ to this new bigger space by the trivial rule $f_n(e_n) = e_n$.

Also, without loss of generality, we assume that the sets $X_n$ are pairwise disjoint (otherwise, substitute $X_n$ by appropriate isometric copies).

We define the requested metric space $X$  as $X = \bigcup_{n =1}^\infty X_n$ equipped with the following metric: on each of $X_n$ the metric is inherited from $X_n$, and for  $x \in X_n$,   $y \in X_m$ with $n \neq m$ we just put $ d(x,y) = \frac32$. It remains to demonstrate that \textbf{(A)} $X$ is not uniformly plastic but \textbf{(B)} it is  strongly plastic.

\textbf{(A)}. For each $n \in \N$ define $g_n : X \to X$ as follows: $g_n(x) = f_n(x)$ for $x \in X_n$, and $g_n(x) = x$ for  $x \in X \setminus X_n$. With this definition, $d(g_n(x), g_n(y)) \ge d(x,y) - \delta_n$ for all $x, y \in X_n$, but there is a pair $x_n, y_n \in X_n$ with $d(g_n(x_n), g_n(y_n)) = d(x,y) + 1$. This demonstrates that $\delta_{X, X}(\eps) \le \delta_n \xrightarrow[n \to \infty]{} 0$ for every $\eps \in (0, 1)$, so  $\delta_{X, X}(\eps) = 0$.

\textbf{(B)}. Let $f : X \to X$ be a noncontractive mapping. Remark, that
$d(f(x), f(e_1)) \ge d(x, e_1) = a_1 + 1$ for all $x \in X_1 \setminus \{e_1\}$, but $d(a, b) < a_1 + 1$ if at least one of $a, b \in X$ does not belong to $X_1$. This implies that $f(X_1) \subset X_1$. By injectivity of $f$ and finiteness of $X_1$ this means that $f(X_1) = X_1$ and the restriction of $f$ to $X_1$ is an isometry.

Passing to complements, we obtain that  $f\left(\bigcup_{n =2}^\infty X_n \right) \subset \bigcup_{n =2}^\infty X_n $, and repeating the same argument we get that $f(X_2) = X_2$ and the restriction of $f$ to $X_2$ is an isometry. Proceeding analogously, we demonstrate that $f$ is a bijective isometry.

\medskip
The above example was constructed artificially. It would be interesting to know whether analogous examples exist among natural examples of EC-plastic spaces. For example:

\begin{prob} \label{prob-sec-appl-unif-1}
It is known that the unit ball $B_{\ell_2}$ of the Hilbert space $\ell_2$ is EC-plastic but not strongly plastic (see Example \ref{example1-not-str-plast} in the Introduction). \emph{Is $B_{\ell_2}$ uniformly plastic? Are there examples of uniformly plastic non-compact ellipsoids in $\ell_2$?} (See \cite{Zav2} for related questions).
\end{prob}

\section{Plasticity and uniform plasticity for pairs of totally bounded spaces} \label{sec-precompact}

The main technical tool that we use in this section is approximation of  a totally bounded metric space $X$ by its finite subsets.

\subsection{Plasticity for pairs of totally bounded spaces} \label{sec-precompact-ssec1} $ $

We use the following notation
$$
\alpha(X,\eps) = \inf\{\sigma(A), \text{ where } A  \text{ is } \eps \text{-net in } X\},
$$
where $\sigma(A)$, like in \eqref{eq-sigma-1}, stands for the sum of pairwise distances between the elements.

In this subsection we present the following model example of possible results.

\begin{theo} \label{sec-pl-tboun-theo1}
Let $X$, $Y$ be totally bounded metric spaces such that $\alpha(X,\eps)\leq \alpha(Y,\eps)$  for every $\eps>0$  and $f\colon X\to Y$ be a non-expansive surjection. Then $f$ is an isometry. In particular, the pair $(Y, X)$ is EC-plastic.
\end{theo}
\begin{proof}
We will argue ``ad absurdum". Suppose  $f$ is not an isometry. Then there are $x_1, x_2\in X$ and $\eps>0$ such that
\begin{equation}\label{ineq1}
d(f(x_1),f(x_2))<d(x_1,x_2)-\eps.
\end{equation}
There also exists finite $\frac{\eps}{6}$-net in $X$ such that
$$
\sigma(A)< \alpha(X,\frac{\eps}{6})+\frac{\eps}{6}.
$$
Since $f$ is surjective, for any $y\in Y$ there is $x\in X$ such that $f(x)=y$.
There exists $a\in A$ such that $d(x,a) < \frac{\eps}{6}$. Then $d(y,f(a)) < \frac{\eps}{6}$. So, $f(A)$ is $\frac{\eps}{6}$-net in $Y$.
There are $a_1,a_2\in A$ such that
$$d(a_1,x_1) < \frac{\eps}{6} \text{ and } d(a_2,x_2) < \frac{\eps}{6}.$$
Then
$$d(f(a_1),f(x_1))<\frac{\eps}{6} \text{ and } d(f(a_2),f(x_2))<\frac{\eps}{6}.$$
At last, we have the following chain of inequalities
\begin{align*}
  d(f(a_1),f(a_2))&\leq d(f(a_1),f(x_1))+d(f(x_1),f(x_2))+d(f(x_2),f(a_2)) \\
  & < d(f(x_1),f(x_2))+\frac{2\eps}{6}<d(x_1,x_2)-\eps+\frac{\eps}{3}\\
  &= d(x_1,x_2)-\frac{2\eps}{3}\leq d(x_1,a_1)+d(a_1,a_2)+d(a_2,x_2)\\
  &-\frac{2\eps}{3}<d(a_1,a_2)-\frac{2\eps}{3}+\frac{2\eps}{6}=d(a_1,a_2)-\frac{\eps}{3}.
\end{align*}
That is why we have
$$
\sigma(f(A))=\sum_{a,b\in A}d(f(a),f(b))<\sum_{a,b\in A}d(a,b)-\frac{\eps}{3}=\sigma(A)-\frac{\eps}{3}.
$$
Then
$\sigma(f(A))<\sigma(A)-\frac{\eps}{3}<\alpha(X,\frac{\eps}{6})+\frac{\eps}{6}-\frac{\eps}{3}=\alpha(X,\frac{\eps}{6})-\frac{\eps}{6}\leq \alpha(Y,\frac{\eps}{6})-\frac{\eps}{6} ,$
which contradicts the definition of $\alpha(Y,\frac{\eps}{6})$.
\end{proof}

\subsection{Uniform strong plasticity for pairs of totally bounded spaces} \label{sec-precompact-ssec2} $ $

A version of Theorem \ref{sec-pl-tboun-theo1} that gives uniform strong plasticity needs much more effort. We managed to achieve this goal using maximal $\eps$-separated sets instead of $\eps$-nets. Recall, that a set $A$ with $|A| \ge 2$ in a metric space $X$ is called \emph{$\eps$-separated}, if all pairwise distances between different elements of $A$ are greater than or equal to $\eps$. For the sake of convenience, a set consisting of one point is considered to be $\eps$-separated for all $\eps > 0$. A subset $A \subset X$ is called \emph{maximal $\eps$-separated}, if it is not a proper subset of any other $\eps$-separated subset of $X$. A maximal $\eps$-separated set is automatically an $\eps$-net. The based on $\eps$-separation substitute for $\alpha(X,\eps)$ that we use below is

$$
s(X,\eps) =\sup\{\sigma(B), \text{ where } B \neq \emptyset  \text{ is a finite } \eps \text{-separated set in } X\}.
$$

\begin{lem}\label{lem-eps-sep}
Let $X$, $Y$ be totally bounded metric spaces,  $\eps>0$, $s(X,\eps) > s(Y,\eps) - \eps$,  and $f\colon X\to Y$ be a non-contractive function. Let $A$, $|A| \ge 2$, be an $\eps$-separated set in $X$ such that
$$
\sigma(A) > s(X,\eps)-\eps.
$$
Then $f(A)$ is a maximal $\eps$-separated set in $Y$.
\end{lem}
\begin{proof}
The $\eps$-separateness of $f(A)$ is a simple consequence of the non-contractiveness of $f$. Let us show $f(A)$ is maximal. Suppose it is not. Then we can add at least one element to $f(A)$ and get $\eps$-separated set $\Tilde{A}$. Then one may get the following contradiction
$$
\sigma(\Tilde{A})>\sigma(f(A))+ 2\eps\geq\sigma(A)+ 2\eps>s(X,\eps) +\eps > s(Y,\eps). \qedhere
$$
\end{proof}

\begin{theo}
Let $X$, $Y$ be totally bounded metric spaces such that for every $\eps>0$ there is $\delta \in (0, \eps)$ such that $s(X,\delta) > s(Y,\delta) - \delta > 0$. Then $(X, Y)$ is a strongly plastic pair.
\end{theo}
\begin{proof}
We will argue ``ad absurdum".  Let $f\colon X\to Y$ be a non-contractive function. Suppose  $f$ is not an isometric embedding. Then there are $x_1, x_2\in X$ and $\eps>0$ such that
\begin{equation}
d(f(x_1),f(x_2))>d(x_1,x_2)+\eps.
\end{equation}
Corresponding to the conditions of the theorem there is $\delta \in (0, \frac{\eps}{6})$ such that $s(X,\delta) > s(Y,\delta) - \delta > 0$. Let us consider  a finite $\delta$-separated set in $X$, such that
$$
\sigma(B)>  \max\{s(X,\delta)-\delta, 0\}.
$$
Since $f$ is non-contractive, and $|B| \ge 2$, Lemma \ref{lem-eps-sep} implies that $f(B)$ is a maximal $\delta$-separated set in $Y$. Thus, it is a $\delta$-net in $Y$, so there are $b_1,b_2\in B$ such that
$$d(f(b_1),f(x_1))<\delta \text{ and } d(f(b_2),f(x_2))<\delta.$$
Then, due to non-contractiveness of $f$
$$d(b_1,x_1)<\delta \text{ and } d(b_2,x_2)<\delta.$$
Hence, one can get the following
\begin{align*}
  d(b_1,b_2)&\leq d(b_1,x_1)+d(x_1,x_2)+d(x_2,b_2)< d(x_1,x_2)+2\delta\\
  &<d(f(x_1),f(x_2))-\eps+2\delta\leq d(f(x_1),f(b_1))\\
  &+d(f(b_1),f(b_2))+d(f(b_2),f(x_2))-\eps+2\delta\\
  &<d(f(b_1),f(b_2))-\eps+4\delta.
\end{align*}
That is why we have
$$
\sigma(B)=\sum_{a,b\in B}d(a,b)<\sum_{a,b\in B}d(f(a),f(b))-\eps+4\delta=\sigma(f(B))-\eps+4\delta.
$$
Then
$\sigma(f(B))>\sigma(B)+\eps-4\delta>s(X,\delta)-\delta+\eps-4\delta=s(X,\delta)+\eps-5\delta\geq s(Y,\delta)+\eps-6\delta > s(Y,\delta)$. This leads to the contradiction with definition of $s(Y,\delta)$.
\end{proof}

From the above theorem one easily gets the following strong plasticity analogue of Theorem \ref{sec-pl-tboun-theo1}.

\begin{cor} \label{sec3-cor-spp}
Let $X$, $Y$ be totally bounded metric spaces such that $s(X,\delta) \ge s(Y,\delta)$ for every $\delta>0$.  Then $(X, Y)$ is a strongly plastic pair.
\end{cor}
\begin{proof}
At first, if $Y$ consists of one element, then a non-contractive function $f\colon X\to Y$ may exist only if $X$ consists of one element as well, and the result is plain. If $|Y| \ge 2$, then $s(Y,\delta) \ge \diam Y$ for every $\delta \in (0, \diam Y)$. Consequently, for every $\eps>0$ every $\delta \in (0, \min\{\eps, \diam Y\})$ satisfies the desired condition $s(X,\delta) > s(Y,\delta) - \delta > 0$.
\end{proof}

Remark, that the mentioned in the introduction classical result about strong plasticity of all totally bounded metric spaces is a particular case of Corollary \ref{sec3-cor-spp}.

It remains to address the uniform strong plasticity of pairs of spaces, which happens to be the most difficult technically piece of work.

For a totally bounded metric space $X$ let us introduce the following notation:
$$
N_{\eps}=N(X,\eps)=\max\{|A|, \text{ where } A  \text{ is } \eps \text{-separated set in } X\}
$$

\begin{lem} \label{sec3-lem2-unif}
Let $X$, $Y$ be totally bounded metric spaces,
and $\eps > 0$ be such that there are $\delta = \delta(\eps) < \min\left\{ \frac{\eps}{N_{\frac{\eps}{9}}(N_{\frac{\eps}{9}}-1)-6}, \frac{\eps}{9(N_{\frac{\eps}{9}}+1)}\right\}$, and $\nu<\frac{\eps}{18}$ with
\begin{equation}\label{eq-red-1}
s\left(X,\frac{\eps}{9}\right)\geq s\left(Y,\frac{\eps}{9}-\delta\right)-\nu > 0.
\end{equation}
Let $f\colon X\to Y$ be a function. Then, if there are $x,y\in X$ with $d(f(x),f(y))\geq d(x,y)+\eps$, then there are $\widetilde{x}, \widetilde{y}\in X$ with $d(f(\widetilde{x}),f(\widetilde{y}))\leq d(\widetilde{x},\widetilde{y})-\delta$.
\end{lem}
\begin{proof}
We will argue by contradiction. Suppose there are $x, y\in X$ such that
\begin{equation}
d(f(x),f(y))\geq d(x,y)+\eps,
\end{equation}
and for every  $\Tilde{x},\Tilde{y}\in X$
\begin{equation}\label{ineq2}
d(f(\Tilde{x}),f(\Tilde{y}))>d(\Tilde{x},\Tilde{y})-\delta.
\end{equation}
There exists a finite $\frac{\eps}{9}$-separated set $A$ in $X$ such that
\begin{equation}\label{eq-red-2}
\sigma(A)>  \max\left\{s\left(X,\frac{\eps}{9}\right)-\mu, 0\right\},
\end{equation}
where $\mu=\frac{\eps}{18}-\nu$.

{\bf Claim.} $f(A)$ is a maximal $(\frac{\eps}{9}-\delta)$-separated set in $Y$.\\
{\it Proof of the Claim.}
Due to our suggestion \eqref{ineq2}, the set $f(A)$ is $(\frac{\eps}{9}-\delta)$-separated. The non-trivial part is to show its maximality. Suppose $f(A)$ is not maximal. Then we can add to $f(A)$ at least one element and get $(\frac{\eps}{9}-\delta)$-separated set $\Tilde{A} \subset Y$. Let us denote by $n$ the quantity of elements in $f(A)$. We have $\frac{\eps}{9}> \delta$ and thus $n = |A|$, since $f$ can not ``glue" the elements of $A$, or else we will get the contradiction with condition \eqref{ineq2}. Observe, $|A| \leq N_{\frac{\eps}{9}}$ and inequality \eqref{eq-red-2} implies that $|A|\geq 2$, so $n \in \left[2, N_{\frac{\eps}{9}}\right]$.
Thus one may get the following chain of inequalities
\begin{align*}
\sigma(\Tilde{A})&\geq\sigma(f(A))+\left(\frac{\eps}{9}-\delta\right)n>\sigma(A)+\left(\frac{\eps}{9}-\delta\right)n-\delta\cdot\frac{n(n-1)}{2}\\
&=\sigma(A)+\frac{\eps}{9}\cdot n-\delta\cdot\frac{n(n+1)}{2}=\sigma(A)+n\cdot\left(\frac{\eps}{9}-\delta\cdot\frac{(n+1)}{2}\right)\\
&> \sigma(A)+n\cdot\left(\frac{\eps}{9}-\frac{(n+1)\eps}{18(N_{\frac{\eps}{9}}+1)}\right)\geq \sigma(A)+n\cdot\left(\frac{\eps}{9}-\frac{\eps}{18}\right)\\
&= \sigma(A)+\frac{n\eps}{18}>s\left(X,\frac{\eps}{9}\right)-\mu+\frac{n\eps}{18}\geq s\left(X,\frac{\eps}{9}\right)+\frac{\eps}{18}+\nu\\
&\geq s\left(Y,\frac{\eps}{9}-\delta\right)-\nu+\nu+\frac{\eps}{18}= s\left(Y,\frac{\eps}{9}-\delta\right)+\frac{\eps}{18}.
\end{align*}
This contradiction completes the proof of the Claim.\\
Now, one may conclude, that $f(A)$ is ($\frac{\eps}{9}-\delta$)-net in $Y$, so, there are $a_1,a_2\in A$ such that
$$d(f(a_1),f(x))<\frac{\eps}{9}-\delta, \text{ and } d(f(a_2),f(y))<\frac{\eps}{9}-\delta.$$
Then, by \eqref{ineq2},
$$d(a_1,x)<\frac{\eps}{9} \text{ and } d(a_2,y)<\frac{\eps}{9}.$$
Finally, one can get the following chain of inequalities, which will help us to obtain the final contradiction
\begin{align*}
  d(a_1,a_2)&\leq d(a_1,x)+d(x,y)+d(y,a_2) < d(x,y)+\frac{2\eps}{9}\\
  &<d(f(x),f(y))-\eps+\frac{2\eps}{9}= d(f(x),f(y))-\frac{7\eps}{9} \\
  &\leq d(f(x),f(a_1))+d(f(a_1),f(a_2))+d(f(a_2),f(y))-\frac{7\eps}{9}\\
  &<d(f(a_1),f(a_2))-\frac{5\eps}{9}-2\delta.
\end{align*}
Now we have
\begin{align*}
\sigma(A)&=\sum_{a,b\in A}d(a,b)<\sum_{a,b\in A}d(f(a),f(b))-\frac{5\eps}{9}-2\delta+\delta\left(\frac{n(n-1)}{2}-1\right)\\
&<\sigma(f(A))-\frac{5\eps}{9}+\frac{\eps(n(n-1)-6)}{2(N_{\frac{\eps}{9}}(N_{\frac{\eps}{9}}-1)-6)}\leq\sigma(f(A))-\frac{\eps}{18}.
\end{align*}
Then
\begin{align*}\sigma(f(A))&>\sigma(A)+\frac{\eps}{18}>s\left(X,\frac{\eps}{9}\right)-\mu+\frac{\eps}{18}\\
&=s\left(X,\frac{\eps}{9}\right)+\nu \stackrel{\mathrm{\eqref{eq-red-1}}}\ge s\left(Y,\frac{\eps}{9}-\delta\right),
\end{align*}
which contradicts the definition of $s\left(Y,\frac{\eps}{9}-\delta\right)$.
\end{proof}

From the above lemma we deduce the main result of the section.

\begin{theo} \label{sec3-theor3-unif}
Let $X$, $Y$ be totally bounded metric spaces such that $s(X,\delta) \ge s(Y,\delta)$ for every $\delta>0$.  Then $(X, Y)$ is a uniformly strongly plastic pair.
\end{theo}
\begin{proof}
Fix  $\eps \in (0, \diam Y)$. Our goal is to prove the existence of $\delta_0 > 0$ with the following property:  for every function $f\colon X\to Y$ if there are $x,y\in X$ with $d(f(x),f(y)) > d(x,y)+\eps$, then there are $\widetilde{x}, \widetilde{y}\in X$ with $d(f(\widetilde{x}),f(\widetilde{y})) < d(\widetilde{x},\widetilde{y})-\delta_0$.

First, let us note that for every $\eps_1 < \eps_2$ one has $s(Y,\eps_2)\leq s(Y,\eps_1)$, since every $\eps_2$-separated set is also $\eps_1$-separated. That is, $s(Y, t)$ is monotone as a function of $t$ and the quantity of its discontinuity points is at most countable.
Thus, we can choose $\eps_0\in (0, \eps)$ in such way that $s(Y, t)$ is continuous at $\frac{\eps_0}{9}$ and consider
$$\nu(\delta)=s\left(Y,\frac{\eps_0}{9}-\delta\right)-s\left(Y,\frac{\eps_0}{9}\right).$$
Then there is $\Delta>0$ such that for every $0<\delta<\Delta$ we have $\nu(\delta)< \frac{\eps_0}{18}$.
Let us choose $\delta_0=\min\left\{\Delta, \frac{\eps_0}{N_{\frac{\eps_0}{9}}\left(N_{\frac{\eps_0}{9}}-1\right)-6},\frac{\eps_0}{9\left(N_{\frac{\eps_0}{9}}+1\right)}\right\}$.
Then for $\eps_0$, $\delta_0$ and $\nu_0=\nu(\delta_0)$ we have the condition
$$
s\left(X, \frac{\eps_0}{9}\right)\geq s\left(Y, \frac{\eps_0}{9}\right)= s\left(Y, \frac{\eps_0}{9}-\delta_0\right)-\nu_0>0,
$$
where the last inequality is guaranteed by condition $\eps_0<\eps<\diam Y$. It remains to apply Lemma \ref{sec3-lem2-unif}.
\end{proof}

The quantitative version of Theorem \ref{sec3-theor3-unif} is Lemma \ref{sec3-lem2-unif}. In some sense it stays not too far from being sharp. This demonstrates the example constructed in Case 1 of Theorem \ref{theo-fin-sharp}. In that example all the pairwise distances between the elements of the space belong to the interval $[a, 2a]$, $a > \eps$, so the space $X$ itself serves as its own finite $\eps$-separated subset consisting of $N = |X|$ elements. So, $s(X, t)$ is constant for $t \le \eps$, and $N_{\frac{\eps}{9}} = N$. In this case for large  $N$ Lemma \ref{sec3-lem2-unif} with $X = Y$ says that every $\delta(\eps) <  \frac{\eps}{N(N-1)-6}$ may be used, but in this example nothing bigger than $ \frac{\eps}{\frac{N}{2}\left(\frac{N}{2}-1\right)}$ can serve as $\delta(\eps)$. So the quotient between the upper and lower estimate for $\delta(\eps)$ is around 4, which is not too bad.

On the other hand, the estimates in the statement of Lemma \ref{sec3-lem2-unif} look ugly, because they relies on the continuity properties of the function $s(Y,\eps)$. It would be nice to have something more elegant in the spirit of Nitka's \eqref{eq-nitka1}.  Of course, it may happen that our generalization with $X \neq Y$ cannot be given in nice terms, similar to those in \eqref{eq-nitka1}. At least, Nitka's proof does not generalize to  $X \neq Y$ because it uses the dynamical system generated by $f: X \to X$ (the sequence of iterations), which makes no sense for $f: X \to Y$.

\subsection{Relationship between strong plasticity and uniform strong plasticity for pairs of totally bounded spaces} \label{sec-precompact-subsection 3} $ $

In Subsection \ref{subsec-finite-applic} we presented an example of strongly plastic  metric space $X$ which was not uniformly strongly plastic. For pairs of different metric spaces such examples can be provided with much less effort, moreover both spaces $X$ and $Y$ in such an example may be chosen to be totally bounded.

\begin{example}  \label{sec-precompact-subsection3-ex1}
Consider the following two subspaces of $\R$ in the standard metric: $X = [0, 1] \cup \{3\}$ and  $Y = [0, 1) \cup \{4\}$. This pair $(X, Y)$ is strongly plastic but not uniformly strongly plastic.
\end{example}
\begin{proof}
1. Strong plasticity. Assume there is a non-contracting map $f: X \to Y$. The point $3 \in X$ has the following property: there are infinitely many points $x \in X$ with $d(3, x) > 2$. Since $f$ is non-contracting, for $f(3) \in Y$ there are also infinitely many points $y \in Y$ with $d(f(3), y) > 2$. The only possibility for this is $f(3) = 4$. Consequently, $f([0, 1]) \subset [0, 1)$, but for a non-contractive map this is impossible, because there should be $d(f(0), f(1)) \ge 1$, but in $[0, 1)$ there are no such points.

2. The absence of uniform strong plasticity is also plain.  Define the functions $f_t: X \to Y$, $t \in (0, 1)$ as follows: $f_t(x) = tx$ for $x \in [0, 1]$, and $f_t(3) = 4$. For these functions and $\eps = 1/2$, $d(f_t(0),f_t(3)) > d(0,3)+\eps$, but for every $\widetilde{x}, \widetilde{y}\in X$ we have $d(f_t(\widetilde{x}),f_t(\widetilde{y})) \ge d(\widetilde{x},\widetilde{y}) - (1 - t)$. So, $\delta_{X, Y}^s(\eps) \le 1-t$ for all $t \in (0, 1)$, that is $\delta_{X, Y}^s(\eps) = 0$.
\end{proof}

For pairs of compact spaces, the situation is different.

\begin{theo} \label{sec3-theor-cpimp-unif}
Let $X$, $Y$ be compact metric spaces such that  $(X, Y)$ is a strongly plastic pair.  Then $(X, Y)$ is a uniformly strongly plastic pair.
\end{theo}
\begin{proof}
Assume to the contrary that $(X, Y)$ is not a uniformly strongly plastic pair. Fix the corresponding  $\eps > 0$ for which $\delta_{X, Y}^s(\eps) = 0$. Then there are functions  $f_n: X \to Y$ and elements $x_n, y_n \in X$, $n = 1, 2, \ldots$ such that
\beq \label{sec3-theor-cpimp-unif-eq0}
d(f_n(x_n), f_n(y_n)) > d(x_n, y_n) + \eps,
\eeq
and for every $\widetilde{x}, \widetilde{y}\in X$
\beq \label{sec3-theor-cpimp-unif-eq1}
d(f_n(\widetilde{x}), f_n(\widetilde{y})) > d(\widetilde{x}, \widetilde{y}) - \frac{1}{n}.
\eeq
Passing, if necessary, to a subsequence, we may additionally assume that the sequences $(x_n), (y_n)$ in $X$ and  $(f_n(x_n)), (f_n(y_n))$ in $Y$ have limits.  Denote $x = \lim_{n \to \infty} x_n$, $y = \lim_{n \to \infty} y_n$, $u = \lim_{n \to \infty} f_n(x_n)$, $v = \lim_{n \to \infty} f_n(y_n)$, and
$\delta_n = d(x_n, x) + d(y_n, y)$. From \eqref{sec3-theor-cpimp-unif-eq0} we have

\beq \label{sec3-theor-cpimp-unif-eq00}
d(u, v) \ge d(x, y) + \eps,
\eeq

Next, we need to consider separately two cases.

\textbf{Case 1: $x = y$.} In this case let us introduce two modified versions of $f_n$: $g_n$ and $h_n$ as follows:
$$
g_n(z)=\begin{cases}
         f_n(z), & \mbox{if } z \notin \{x_n, x\} \\
         f_n(x), & \mbox{if } z = x_n \\
         f_n(x_n), & \mbox{if } z = x,
       \end{cases}
$$
and analogously
$$
h_n(z)=\begin{cases}
         f_n(z), & \mbox{if } z \notin \{y_n, x\} \\
         f_n(x), & \mbox{if } z = y_n \\
         f_n(y_n), & \mbox{if } z = x.
       \end{cases}
$$
By the triangle inequality, \eqref{sec3-theor-cpimp-unif-eq1} implies that  for every $\widetilde{x}, \widetilde{y}\in X$
\beq \label{sec3-theor-cpimp-unif-eq11}
d(g_n(\widetilde{x}), g_n(\widetilde{y})) > d(\widetilde{x}, \widetilde{y}) - \frac{1}{n} - \delta_n, \textrm{ and}
\eeq
\beq \label{sec3-theor-cpimp-unif-eq111}
d(h_n(\widetilde{x}), h_n(\widetilde{y})) > d(\widetilde{x}, \widetilde{y}) - \frac{1}{n} - \delta_n.
\eeq

Fix a free ultrafilter $\F$ on $\N$ and consider the $\F$-pointwise limits $g, h$ of $(g_n)$ and $(h_n)$ respectively:
$$
g(z) = \lim_\F g_n(z), \quad h(z) = \lim_\F h_n(z).
$$
(For a brief introduction to filters, ultrafilters and compactness we refer to \cite[Section 16.1]{Kad2018}).
The conditions \eqref{sec3-theor-cpimp-unif-eq11} and \eqref{sec3-theor-cpimp-unif-eq111} imply that $g$ and $h$ are non-contractive, so by the strong plasticity of $(X,Y)$,  $g$ and $h$ are isomeric embeddings. By the construction, for $z \neq x$ we have
\beq \label{sec3-theor-cpimp-unif-eqsmall}
g(z) = h(z) = \lim_\F f_n(z).
\eeq
 This implies that $g(x) = h(x)$. Indeed, if $x$ is not an isolated point, then the needed condition is just the limiting case of \eqref{sec3-theor-cpimp-unif-eqsmall} as $z \to x$; and if $x$ is  an isolated point, then all $x_n$ and $y_n$, up to finitely many of them, are equal to $x$, so $g_n(x) = h_n(x)$ for all but finitely many values of $n \in \N$.

On the other hand, $g(x) = \lim_\F f_n(x_n) = u$, $h(x) = \lim_\F f_n(y_n) = v$  and $u \neq v$ thanks to \eqref{sec3-theor-cpimp-unif-eq00}.

\textbf{Case 2: $x \neq y$.}  In this case, without loss of generality, we may assume that also $x_n \neq y_n$ for each $n \in \N$. This enables us to combine the formulas for $g_n$ and $h_n$ and define functions $w_n: X \to Y$
$$
w_n(z)=\begin{cases}
         f_n(z), & \mbox{if } z \notin \{x_n, x, y_n, y\} \\
         f_n(x), & \mbox{if } z = x_n \\
         f_n(x_n), & \mbox{if } z = x \\
         f_n(y), & \mbox{if } z = y_n \\
         f_n(y_n), & \mbox{if } z = y.
       \end{cases}
$$
Again, fix a free ultrafilter $\F$ on $\N$ and consider the limiting function
$$
w(z) = \lim_\F w_n(z).
$$
Once more, by the triangle inequality, \eqref{sec3-theor-cpimp-unif-eq1} implies that  for all $\widetilde{x}, \widetilde{y}\in X$
$$
d(w_n(\widetilde{x}), w_n(\widetilde{y})) > d(\widetilde{x}, \widetilde{y}) - \frac{1}{n} - \delta_n,
$$
so $w$ is non-contractive. By the strong plasticity of $(X,Y)$, this means that $w$ is an isomeric embedding. On the other hand,
$$
d(w(x), w(y)) =  \lim_\F d(f_n(x_n), f_n(y_n)) \stackrel{\mathrm{\eqref{sec3-theor-cpimp-unif-eq0}}}\ge  \lim_\F d(x_n, y_n) + \eps = d(x, y) + \eps,
$$
which gives the desired contradiction.
\end{proof}

The following question remains open for us.

\begin{prob} \label{prob-sec3-cpimp-unif}
Do there exist metric spaces $X$, $Y$ such that $X$ is totally bounded, $Y$ is compact, and the corresponding pair $(X, Y)$ is strongly plastic but not uniformly strongly plastic?
\end{prob}

\section*{Acknowledgements}
The authors are grateful to Rainis Haller and Nikita Leo for useful discussions on the subject.

\bibliographystyle{amsplain}
\end{document}